\documentclass[11pt]{amsart}

\usepackage{graphicx, graphics}

\usepackage{hyperref}

\DeclareMathAlphabet{\mathpzc}{OT1}{pzc}{m}{it}

\emergencystretch=\maxdimen
\usepackage{amsmath}
\usepackage{amsthm}
\usepackage{amsfonts}
\usepackage{color}
\usepackage[english]{babel}
\usepackage[margin=1.5in]{geometry}

\usepackage{enumerate}
\usepackage{subcaption}

\newcommand{\red}{\color{red}}

        \definecolor{brown}{rgb}{1,0,1}

\usepackage[latin9]{inputenc}
\usepackage{amsmath}
\usepackage{amsfonts}
\usepackage{amssymb}
\usepackage{amsthm}

\numberwithin{equation}{section}

\newtheorem{theo}{Theorem}[section]
\newtheorem{prop}[theo]{Proposition}

\theoremstyle{definition}
\newtheorem{def1}[theo]{Definition}
\theoremstyle{remark}
\newtheorem{rema}[theo]{Remark}

\newcommand{\N}{{\mathbb N}}
\newcommand{\nwc}{\newcommand}
\nwc{\eps}{\epsilon}
\nwc{\vareps}{\varepsilon}
\nwc{\Oph}{\operatorname{Op}_\hbar}
\nwc{\ra}{\rangle}
\nwc{\la}{\lambda}

\nwc{\mf}{\mathbf} 
\nwc{\blds}{\boldsymbol} 
\nwc{\ml}{\mathcal} 

\nwc{\defeq}{\stackrel{\rm{def}}{=}}

\nwc{\cE}{\ml{E}}
\nwc{\cN}{\ml{N}}
\nwc{\cO}{\ml{O}}
\nwc{\cP}{\ml{P}}
\nwc{\cU}{\ml{U}}
\nwc{\cV}{\ml{V}}
\nwc{\cW}{\ml{W}}
\nwc{\tU}{\widetilde{U}}
\nwc{\IN}{\mathbb{N}}
\nwc{\IR}{\mathbb{R}}
\nwc{\IZ}{\mathbb{Z}}
\nwc{\IC}{\mathbb{C}}
\nwc{\IT}{\mathbb{T}}
\nwc{\tP}{\widetilde{P}}
\nwc{\tPi}{\widetilde{\Pi}}
\nwc{\tV}{\widetilde{V}}
\nwc{\supp}{\operatorname{supp}}
\nwc{\rest}{\restriction}

\newcommand{\R}{{\mathbb R}}

\newcommand{\Z}{{\mathbb Z}}

\renewcommand{\H}{{\mathbf H}}

\renewcommand{\phi}{\varphi}

\newcommand{\ep}{\varepsilon}

\newtheorem{cor}[theo]{{\sc Corollary}}
\newtheorem{lem}[theo]{{\sc Lemma}}

\title [Quantum ergodicity and $L^p$ restrictions] {Quantum ergodicity and $L^p$ norms of restrictions of eigenfunctions}

\author{Hamid Hezari }
\address{Department of Mathematics, UC Irvine, Irvine, CA 92617, USA} \email{hezari@math.uci.edu}

\begin{document}


\maketitle

\begin{abstract}
We prove an analogue of Sogge's local $L^p$ estimates \cite{So15} for $L^p$ norms of restrictions of eigenfunctions to submanifolds, and use it to show that for quantum ergodic eigenfunctions one can get improvements of the results of Burq-G\'erard-Tzvetkov \cite{BuGeTz}, Hu \cite{Hu}, and Chen-Sogge \cite{ChSo}. The improvements are logarithmic on negatively curved manifolds (without boundary) and by $o(1)$ for manifolds (with or without boundary) with ergodic geodesic flows. In the case of ergodic billiards with piecewise smooth boundary, we get $o(1)$ improvements on $L^\infty$ estimates of Cauchy data away from a shrinking neighborhood of the corners, and as a result using the methods of \cite{GRS, ZeJJ1, ZeJJ}, we get that the number of nodal domains of $2$-dimensional ergodic billiards tends to infinity as $\la \to \infty$.  These results work only for a full density subsequence of any given orthonormal basis of eigenfunctions.
\\

\noindent We also present an extension of the $L^p$ estimates of \cite{BuGeTz, Hu, ChSo} for the restrictions of Dirichlet and Neumann eigenfunctions to  compact submanifolds of the interior of manifolds with piecewise smooth boundary.  This part does not assume ergodicity on the manifolds. 
\end{abstract}

\section{Introduction}

The quantum ergodicity results of \cite{Sh, CdV, Ze87}, and the small scale quantum ergodicity results of \cite{Han, HeRi} on negatively curved manifolds\footnote{See \cite{LuSa, Yo, HeRiTorus, LeRu} for parallel results in the arithmetic setting.} are recently shown to imply improvements for several measurements of eigenfunctions such as: $L^p$ norms \cite{HeRi, So15}, number of nodal domains \cite{ZeN}, growth rates, the size of nodal sets, order of vanishing of eigenfunctions \cite{He16a}, and also the inner radius of nodal domains \cite{He16b}. The purpose of this article is to prove  another application of $L^2$ equidistribution of eigenfunctions.  We will show that $L^p$ estimates of restrictions of eigenfunctions to submanifolds  can be improved according to certain powers of $r(\la)$, where $r(\la)$ is the least radius  of shrinking geodesic balls on which the eigenfunctions equidistribute uniformly.

Let $(X,g)$ be a compact connected smooth Riemannian manifold of dimension $n \geq 2$, with or without boundary. When $\partial X \neq \emptyset$, we assume that $\partial X$ is piecewise \footnote{See Section \ref{PS2} for the precise definition.} smooth. We denote the singular part of the boundary by $\mathcal S$.  Let $\Sigma$ be a compact smooth submanifold \footnote{ $\Sigma$ can have a boundary, in which case we assume its boundary is smooth.} of dimension $k$ of the interior $X \backslash \partial X$, or of the regular part of the boundary $\partial X  \backslash \mathcal S$.  The metric $g$ induces natural volume measures on $X$ and $\Sigma$, which we denote by $d_gv$ and $d_g \sigma$ respectively and denote $L^p(X)$ and $L^p(\Sigma)$ to be the corresponding $L^p$ spaces.   Suppose $\Delta_g$ is the positive Laplace-Beltrami operator on $(X, g)$ (with Dirichlet or Neumann boundary conditions if $\partial X \neq \emptyset$).  For $\lambda \geq 1$, we define the spectral cluster operators $\Pi_\lambda=\bigoplus_{ \sqrt{\lambda_j} \in\,[\sqrt{\lambda}, \sqrt{\lambda}+1]} \Pi_{E_{\lambda_j}}$, where $\Pi_{E_{\lambda_j}}$ is the orthogonal projection operator onto the eigenspace $E_{\lambda_j}=\, \text{ker}\, (\Delta_g - \lambda_j)$. 

\subsection{\textbf{Main tool: local $L^p$ restriction estimates} } 

The main lemma  that enables us to prove our results is the following analogue of Sogge's local $L^p$ estimates \cite{So15} for local $L^p$ norms of restrictions of eigenfunctions of $\Delta_g$ to $\Sigma$. 

\begin{lem} \label{LocalLp}Let $(X,g)$ and $\Sigma$ be as above and assume $\partial X$ is smooth if it is non-empty. Also let $2 \leq p \leq \infty$ and $r_0 = \text{inj}(X, g)$. Suppose there exists $\delta=\delta(n, k, p)$ and $C >0$ independent of $\lambda$, such that for all $\lambda \geq 1$
\begin{equation}\label{NormOfSpectralCluster}
||\Pi_\lambda ||_{L^2(X) \to L^p(\Sigma)} \leq C \lambda^{\delta}.
\end{equation}
Then for any eigenfunction $\psi_{\lambda_j} \in E_{\lambda_j}$ with $\lambda_j \geq 1$, any $x \in \Sigma$, and any $r \in [\lambda_j^{-1/2}, r_0/2]$, there exists a constant $C_1$, dependent only on $C$ and $\delta$, such that
\begin{equation} \label{LocalLpEstimates}
||\psi_{\lambda_j} ||_{L^p(B_{r/2}(x) \cap \Sigma)} \leq C_1 r^{-\frac{1}{2}} || \psi_{\lambda_j}||_{L^2(B_{r}(x))} \lambda_j^\delta . 
\end{equation}
\end{lem}

Our first result is the following conditional theorem which will  be obtained using the above lemma and a covering argument. 
\begin{theo} \label{LpEstimatesSS} Let $(X,g)$ and $\Sigma$ be as in Lemma \ref{LocalLp} and $2 \leq p \leq \infty$.  Suppose there exists $\delta=\delta(n, k, p)$ and $C$ independent of $\lambda$, such that for all $\lambda \geq 1$
\begin{equation}\label{NormOfSpectralCluster1}
||\Pi_\lambda ||_{L^2(X) \to L^p(\Sigma)} \leq C \lambda^{\delta}.
\end{equation} Let $\psi_{\lambda_j} \in E_{\lambda_j}$ be an $L^2$(X)-normalized eigenfunction with $\lambda_j \geq 1$. 
Then there exists $r^*$ dependent only on $(X, g)$ and $\Sigma$, such that if for some $r \in [\lambda_j^{-1/2}, r^*]$ and for all geodesic balls $\{B_{r}(x)\}_{x \in \Sigma}$ we have
\begin{equation} \label{L2assumption2}
|| \psi_{\la_j}||^2_{L^2(B_{r}(x))} \leq K r^n,
\end{equation} for some constant $K$ independent of $x$, 
then 
$$|| \psi_{\la_j}||_{L^p(\Sigma)}  \leq C_2 r^{\frac{n-1}{2}-\frac{k}{p}} \lambda_j^\delta.$$ Here $C_2$, depends only on $(X, g)$, $\Sigma$, p, $K$, and $C$. 
\end{theo}

 As it is evident the above theorem is subject to two conditions. Under (\ref{NormOfSpectralCluster1}) one gets the \textit{trivial} $L^p(\Sigma)$ bounds $\la_j^ \delta$, but under (\ref{L2assumption2}) one can improve the trivial bounds $\la_j^\delta$ by the small \footnote{When $(k, p) =(n-1, 2)$ this small factor does not appear.} factor $r^{\frac{n-1}{2}-\frac{k}{p}}$. The condition (\ref{NormOfSpectralCluster1}) on $\Pi_\la$ is proved in \cite{BuGeTz, Hu, ChSo} for smooth manifolds without boundary. The condition (\ref{L2assumption2}) on $L^2$ norms on small balls is proved in \cite{Han, HeRi}  for negatively curved manifolds with $r= (\log \la_j)^{- \kappa}$, for any $\kappa \in (0, \frac{1}{2n})$.  Putting these results and the above theorem together we obtain the following logarithmic improvements:

\begin{theo}\label{LpQENC}

Let $(X, g)$ be a boundaryless compact connected smooth Riemannian manifold of dimension $n \geq 2$, with negative sectional curvature, and let $\Sigma$ be a compact submanifold of $X $ of dimension $k$. Let $\epsilon >0$ and let $\{ \psi_{\lambda_j}\}_{j \in \IN}$ be any ONB of $L^2(X)$ consisting of eigenfunctions of $\Delta_g$ with eigenvalues $\{\lambda_j\}_{j \in \IN}$. Then there exists $S \subset \IN$ of full density such that with the exception of $(k, p) =(n-1, 2)$ $$ 2 \leq p \leq \infty, j \in S: \quad || \psi_{\lambda_j}||_{L^p(\Sigma)}  = O_{\epsilon}\left ( (\log \lambda_j)^{-(\frac{1}{2n})(\frac{n-1}{2}-\frac{k}{p})+\epsilon}\lambda_j^{\delta(n, k, p)} \right ),$$  where 
\begin{equation} \label{delta}
\Small {\delta(n,k,p)=   \begin{cases}    \frac{n-1}{8}-\frac{n-2}{4p}  & k=n-1, \;  2 \leq p < \frac{2n}{n-1} \;  (\text{\cite{BuGeTz}}) \\
                                                \frac{n-1}{4}-\frac{n-1}{2p}  & k=n-1, \;   \frac{2n}{n-1}  \leq p \leq \infty \; (\text{\cite{BuGeTz}, and \cite{Hu} for} \; p=\frac{2n}{n-1} )\\
                                               \frac{n-1}{4}-\frac{k}{2p} & k < n-2, \;   2 \leq p \leq \infty, \; \text{or} \; k=n-2, \;2<p \leq \infty \; (\text{\cite{BuGeTz}}) \\
                                                \frac{1}{4} & k=n-2, n=3 , p=2 \; \; (\text{\cite{ChSo}}) \\
                                               \text{see Remark \ref{log}} & k=n-2, n>3, p=2. \\
\end{cases} }
\end{equation}
When $k=n-1$ and $\Sigma$ has a non-zero scalar second fundamental form, the exponent $\delta(n,n-1,p)$ can be improved to 
\begin{equation} \label{deltatilde} \Small \tilde{\delta}(n,p)= \frac{n-1}{6}- \frac{2n-3}{6p}, \qquad 2 < p < \frac{2n}{n-1}. \qquad  (\text{\cite{BuGeTz}}) \end{equation}
\end{theo}

In the the case of manifolds with ergodic geodesic flows (ergodic billiard flows when the manifold has a boundary), the improvements are given only by $o(1)$.  We emphasize that the following theorem is only valid for interior submanifolds.

\begin{theo}\label{LpQE}
Let $(X, g)$ be a compact connected smooth Riemannian manifold of dimension $n \geq 2$, possibly with (piecewise smooth) boundary, and let $\Sigma$ be a compact submanifold of $X \backslash \partial X$ of dimension $k$. Suppose the geodesic flow on $(X, g)$ is ergodic when $\partial X = \emptyset $, or suppose the billiard flow on $(X,g)$ is ergodic when $\partial X \neq \emptyset$.  Let $\{ \psi_{\lambda_j} \}_{j \in \IN}$ be any ONB of $L^2(X)$ consisting of eigenfunctions of $\Delta_g$ (with Dirichlet or Neumann boundary conditions when $\partial X \neq \emptyset$) with eigenvalues $\{\lambda_j\}_{j \in \IN}$. Then there exists $S \subset \IN$ of full density such that, with the exception of $(k, p) =( n-1, 2)$, we have  $$2 \leq p \leq \infty, j \in S: \quad || \psi_{\lambda_j}||_{L^p(\Sigma)}  = o\left (\lambda_j^{\delta(n, k, p)}\right ),$$ where $\delta$ is defined by (\ref{delta}). 
When $k=n-1$ and $\Sigma$ has a non-zero scalar second fundamental form, the exponent $\delta(n,n-1,p)$ can be improved to $\tilde \delta$ denoted in (\ref{deltatilde}). 
\end{theo}

\subsection{Improved sup norms and the number of nodal domains} When $\partial X \neq \emptyset$,  we get the following $o(1)$ improvements of the results of Grieser \cite{Gr02}, Sogge \cite{So02}, and \cite{Xu} for $L^\infty$ norms of QE eigenfunctions and their gradient, away from a shrinking neighborhood $\mathcal T_{\ep_j} ( \mathcal S)= \{ x \in X; \; d(x, \mathcal S) \leq \ep_j \}$ of the singular part of the boundary.  In particular if $\Sigma \subset \partial X \backslash \mathcal S$ is a smooth compact submanifold of the regular part of the boundary, we get sup norms of the form $ o (\lambda_j^{\frac{n-1}{4}})$ for the boundary traces of QE eigenfunctions on $\Sigma$. 

\begin{theo}\label{LinftyQE}
Let $(X, g)$ be a compact connected smooth Riemannian manifold of dimension $n \geq 2$ with piecewise smooth boundary and let $\mathcal S$ be the singular part of $\partial X$. Suppose the billiard flow on $(X,g)$ is ergodic.  Let $\{ \psi_{\lambda_j} \}_{j \in \IN}$ be any ONB consisting of eigenfunctions of $\Delta_g$ (with Dirichlet or Neumann boundary conditions) with eigenvalues $\{\lambda_j\}_{j \in \IN}$. Then there exist $S \subset \IN$ of full density and $\{ \ep_j \}_{j \in S}$ with $\ep_ j \to 0^+$, such that for $j \in S$
$$ \sup_{X \backslash \mathcal T_{\ep_j}(\mathcal S)} | \psi_{\la_j} |  = o\left (\lambda_j^{\frac{n-1}{4}}\right ) \quad \text{and} \quad  \sup_{X \backslash \mathcal T_{\ep_j}(\mathcal S)}  \la_j^{- \frac12} |\nabla \psi_{\la_j} |  = o\left (\lambda_j^{\frac{n-1}{4}}\right ).$$
Hence, in particular 
$$\text{in the Neumann case:} \quad  \sup_{\partial X \backslash \mathcal T_{\ep_j}(\mathcal S)} | \psi_{\la_j} |  = o\left (\lambda_j^{\frac{n-1}{4}}\right ), $$
$$ \qquad \text{in the Dirichlet case:} \quad   \sup_{\partial X \backslash \mathcal T_{\ep_j}(\mathcal S)}  \la_j^{- \frac12}|\partial_n \psi_{\la_j} |  = o\left (\lambda_j^{\frac{n-1}{4}}\right ).$$
\end{theo}

As a corollary of the above sup norm estimates, and using the method \footnote{See \cite{JaJu} for a different technique applied to even and odd QE eigenfunctions of surfaces with an isometric involution.} of Jung-Zelditch \cite{ZeJJ}, we get the following generalization of their results on the number of nodal domains for any ergodic billiard table with piecewise smooth boundary, including the Bunimowich stadium \cite{Bun}, Sinai dispersive billiards \cite{Si}, and their families. 

\begin{theo}\label{NN} Let $(X, g)$ be a compact connected smooth Riemannian manifold  with piecewise smooth boundary of dimension $n=2$. Suppose the billiard flow on $(X,g)$ is ergodic.  Let $\{ \psi_{\lambda_j} \}_{j \in \IN}$ be any ONB consisting of eigenfunctions of $\Delta_g$ (with Dirichlet or Neumann boundary conditions) with eigenvalues $\{\lambda_j\}_{j \in \IN}$. Then there exists $S \subset \IN$ of full density such that the number of nodal domains of $\psi_{\la_j}$ tends to infinity as $ \la_j \to \infty$  along $S$. 
\end{theo}

We recall that in \cite{ZeJJ}, the above result was proved for non-positively curved manifolds with smooth concave boundary. It is known (see \cite{Si}, \cite{ChSi}, \cite{BuChSi}) that such billiard tables are ergodic. The main ingredients of the proof of \cite{ZeJJ} are quantum ergodicity theorems of \cite{GeLe, Bu, HaZe} for the boundary values of eigenfunctions, the so called ``Kuznecov sum formula" for manifolds with smooth boundary \cite{HHHZ}, and sup norm estimates of size $o(\la^{1/4})$ for the boundary values of eigenfunctions on positively curved manifolds with smooth concave boundary \cite{SoZe14}. For us to prove the above theorem, the first two ingredients are still available except that our boundary can have singular points. We will discuss in the proof that as long as we stay away from $\mathcal S$, the corners will not cause any problems. The last ingredient, which is the new ingredient, is the sup norms in Theorem \ref{LinftyQE} that this paper provides.

\subsection{ $L^p$ restrictions for manifolds with boundary}
In this section we show that as a corollary of the local $L^p$ restrictions estimates (Lemma \ref{LocalLp}),  one can extend the results of \cite{BuGeTz, Hu, ChSo} to manifolds with boundary. 

\begin{cor} \label{BoundaryCor} Let $(X, g)$ be a compact connected smooth Riemannian manifold of dimension $n$, with (piecewise smooth) boundary, and let $\Sigma$ be a compact smooth submanifold of $X \backslash \partial X$ of dimension $k$. Suppose $\psi_\la$ is an $L^2(X)$-normalized eigenfunction of $\Delta_g$ with Dirichlet or Neumann boundary conditions. Then for all $ 2 \leq p \leq \infty$
$$ || \psi_\la ||_{L^p( \Sigma)} \leq C_s \la^{\delta(n, k, p)},$$ where $C_s$ depends on $s= d( \Sigma, \partial X)$ and $\delta(n, k, p)$ is defined by (\ref{delta}). When $k=n-1$ and $\Sigma$ has a non-zero scalar second fundamental form, the exponent $\delta(n,n-1,p)$ can be improved to $\tilde \delta(n, p)$ denoted in (\ref{deltatilde}).
\end{cor}
We emphasize that the above corollary holds for the eigenfunctions and not necessarily for the spectral cluster operators $\Pi_\la$. See the results of Blair\cite{Bla} where estimates on $||\Pi_\la ||_{L^2(X) \to L^p(\Sigma)}$ are given with a natural loss due to whispering gallery modes. We also underline that in the proof of this corollary, in addition to Lemma \ref{LocalLp}, we heavily use the results of \cite{BuGeTz, Hu, ChSo} on the spectral cluster operators $\Pi_\la$ on compact manifolds without boundary.

\subsection{\textbf{Remarks}}
\begin{rema}\label{Chen}
We point out that when $(X,g)$ is negatively curved, for $p= \infty$ the estimate in Theorem \ref{LpQENC} is worse than Berard's upper bound $ (\log \lambda_j)^{-\frac{1}{2}}{\lambda_j^{\frac{n-1}{4}}}$, which is valid on all non-positively curved compact manifolds \cite{Be}. In addition, for $k=n-1$ and the range $ p > \frac{2n}{n-1}$, and also $k \leq n-2$ and $ p > 2$, our estimates are weaker than the upper bounds $(\log \lambda_j)^{-\frac{1}{2}}{\lambda_j^{\delta(n, k, p)}}$ of \cite{Ch}.
\end{rema}
\begin{rema}\label{QER} Our improved estimates do not include the case $(k, p)=(n-1, 2)$. However, by the quantum restriction theorems of \cite{TZ} and \cite{DZ},  under a certain lack of microlocal symmetry assumption on $\Sigma$, the restrictions $\psi_\lambda |_\Sigma$ are QE on $\Sigma$. Hence in particular $ || \psi_\lambda ||^2_{L^2 (\Sigma)}$ are uniformly bounded. Of course one can interpolate these uniform upper bounds for $p=2$ with our estimates for $p=\frac{2n}{n-1}$ to get better estimates, however the full density subsequences that arise in the proofs of  \cite{TZ} and \cite{DZ} depend on $\Sigma$, while the full desity subsequences that are chosen in our Theorems \ref{LpQE} and \ref{LpQENC} are independent of the choice of $\Sigma$.
\end{rema}

\begin{rema}\label{log} The exponents $\delta(k, n, p)$  in (\ref{delta}), and  $\tilde{\delta}(n,p)$  in (\ref{deltatilde}) when $\Sigma$ is \textit{curved}, hold more generally for the spectral cluster operators (\ref{NormOfSpectralCluster}) and are sharp by \cite{BuGeTz} for the given range of parameters in (\ref{delta}). The only case not covered in (\ref{delta}) is the case $k=n-2$, $p=2$. It was proved in \cite{BuGeTz} that in this case for $\la \geq 2$
$$||\Pi_\lambda ||_{L^2(X) \to L^2(\Sigma)} \leq C \sqrt{ \log \la} \, \lambda^{\frac{1}{4}}.$$ It is believed that this estimate is not sharp and one should be able to remove the $\log$ term as was proved in \cite{ChSo} for $n=3$. 
\end{rema}

\begin{rema}
One might be able to generalize the $o(1)$ improvements in Theorem \ref{LpQE} for semiclassical pseudodifferential operators whose principal symbols generate ergodic Hamiltonian systems. However, we point out that the local $L^p$ estimates we obtained used the finite speed propagation property of the even part of the wave group, something that we do not necessarily have in a more general framework.  It is possible that a rescaling argument applied to the results of \cite{Ta12} and \cite{HaTa} would give local $L^p$ estimates for some range $r > h^\alpha$, which can still  be useful and for example would give $o(1)$ improvements for QE eigenfunctions of semiclassical operators. In fact, such a rescaling argument with the help of  semiclassical $L^p$ estimates of \cite{KTZ} was used in \cite{HeRi} to obtain local $L^p$ estimates, however the proof of \cite{HeRi} (although not explicitly stated) only worked for $r \geq \lambda^{-1/8}$. In \cite{So15}, Sogge gave an elegant proof that works for $r \geq \lambda^{-1/2}$, which is the method we have adapted in the present paper.  
\end{rema}

\begin{rema} Using the small-scale QE results of \cite{LeRu}, and our Theorem \ref{LpEstimatesSS}, we can get polynomial improvements for toral eigenfunctions in dimensions $n \geq 3$. 

\end{rema}
\begin{rema} One can also use the result of Blair \cite{Bla} to prove $o(1)$ improvements on his $L^p$ estimates for the boundary values of Neumann eigenfunctions, analogous to the sup norm estimates of Theorem \ref{LinftyQE}. We underline that in this situation there will be a loss for $p \neq \infty$ due to whispering gallery modes near the boundary.
\end{rema}

\begin{rema} Using our method and also the spectral cluster estimates of Smith-Sogge \cite{SoSm} on manifolds with smooth boundary, one can give $o(1)$ improvements on $L^p$ norms of QE eigenfunctions on  manifolds with piecewise smooth boundary, away from a shrinking neighborhood of the corners. 

\end{rema}

\begin{rema} As a final remark, we must mention the recent works of Blair-Sogge \cite{BlSo}, Xi-Zhang \cite{XZ}, and Marshall \cite{Ma}. In \cite{BlSo}, logarithmic improvements are given on the $L^2$ norms of restrictions to geodesics on non-positively curved surfaces.  In \cite{XZ},  inspired by the works of Blair-Sogge \cite{BlSo} and Sogge \cite{So16b} $\log \log$ improvements are obtained for $L^4$ geodesic restrictions on non-positively curved surfaces and $\log$ improvements in the case of compact hyperbolic surfaces.  In \cite{Ma}, $L^2$ geodesic restriction estimates are improved by a power of $\la$ for Hecke-Maass cusp forms. 

\end{rema}

\subsection{\label{SSQE} \textbf{Background on small scale quantum ergodicity}} First, we recall that the quantum ergodicity result of Shnirelman-Colin de Verdi\`ere-Zelditch \cite{Sh, CdV, Ze87} implies in particular that if the geodesic flow of a smooth compact Riemannian manifold without boundary is ergodic then for any ONB $\{ \psi_{\lambda_j} \}_{j=1}^\infty$ consisting of the eigenfunctions of $\Delta_g$, there exists a full density subset $S \subset \IN$ such that for any fixed $r < \text{inj}(X, g)$, independent of $\lambda_j$, one has \begin{equation}\label{QE} || \psi_{\lambda_j}||^2_{L^2(B_{r}(x))} \sim \frac{\text{Vol}_g(B_r(x))}{\text{Vol}_g(X)}, \qquad \text{as}\quad  \lambda_j \to \infty, \quad j \in S.  \end{equation}
The analogous result on manifolds with piecewise smooth boundary and with ergodic billiard flows was proved by \cite{ZZ}.  

The small scale equidistribution  problem asks whether (\ref{QE}) holds for $r$ dependent on $\lambda_j$. A quantitative QE result of Luo-Sarnak \cite{LuSa} shows that the Hecke eigenfunctions on the modular surface satisfy this property along a density one subsequence for $r=\lambda^{- \kappa}$ for some small $\kappa>0$.  Also, under the generalized Riemann hypothesis, Young \cite{Yo} has proved that small scale equidistribution holds for Hecke eigenfunctions for $r = \lambda^{-1/4 + \epsilon}$.  In \cite{Han} and \cite{HeRi}, this problem was studied for the eigenfunctions of compact negatively curved manifolds. To be precise, it was proved that on compact negatively curved manifolds without boundary, for any $\epsilon >0$ and any ONB $\{ \psi_{\lambda_j} \}_{j=1}^\infty$  consisting of the eigenfunctions of $\Delta_g$, there exists a subset $S_{\epsilon} \subset \IN$ of full density such that for all $x \in X$ and $j \in S_{\epsilon}$:
\begin{equation} \label{QENC} \quad K_1 r^n \leq  || \psi_{\lambda_j}||^2_{L^2(B_{r}(x))} \leq K_2 r^n, \qquad \text{with} \;\; r=(\log \lambda_j)^{-\frac{1}{2n} +\epsilon}, \end{equation} for some positive constants $K_1, K_2$ which depend only on $(X, g)$ and $\epsilon$. \footnote{The same result was proved in \cite{Han} for $r=(\log \lambda_j)^{-\frac{1}{3n} +\epsilon}$.} 

We also point out that although eigenfunctions on the flat torus $\R^n / \Z^n$  are not quantum ergodic, however they equidistribute on the configuration space $\R^n / \Z^n$ (see \cite{MaRu}, and also \cite{Ri} and \cite{Taylor} for later proofs). So one can investigate the small scale equidistribution property for toral eigenfunctions. It was proved in \cite{HeRiTorus} that a commensurability of $L^2$ masses such as (\ref{QENC}) is valid for a full density subsequence with $r = \lambda^{-1/(7n+4)}$. Lester-Rudnick \cite{LeRu} improved this rate of shrinking to  $r= \lambda^{- \frac{1}{2n-2} +\epsilon}$, and in fact they proved that the stronger statement (\ref{QE}) holds.  They also showed that their results are almost sharp \footnote{ That it fails for $r= \lambda^{- \frac{1}{2n-2} \red{-} \epsilon}$ for a positive density subsequence of some ONB.}. The case of interest is $n=2$, which gives $r= \lambda^{-1/2 +\epsilon}$. A natural conjecture is that this should be the optimal rate \footnote{We must mention a result of Ingremeau \cite{In}, where small scale QE is proved for any  $r \gg \la^{-1/2}$ for distorted plane waves on non-compact non-positively curved manifolds with Euclidean ends.} of shrinking on negatively curved manifolds. A recent result of \cite{Han16} proves that random eigenbases on the torus enjoy small scale QE for $r= \lambda^{-\frac{n-2}{4n} +\epsilon}$, which is better than \cite{LeRu} for $n \geq 5$. 

\subsection{Definition of manifolds with piecewise smooth boundary} \label{PS2}

We follow the definition of \cite{HaZe}, but we allow our Riemannian manifolds to be non-Euclidean.

\begin{def1}\label{PS} Let $X$ be the closure of an open connected subset of a smooth compact connected boundaryless manifold $\tilde X$ of dimension $n \geq 2$. We say that $X \subset \tilde X$ is a piecewise smooth manifold if the boundary $\partial X$ is Lipschitz, and can be written as a finite disjoint union
$$
\partial X = H_1 \cup \dots \cup H_m \cup \mathcal S,
$$
where each $H_i$ is an open subset of a smooth embedded hypersurface $S_i$, with $\text{int}(X)$ lying locally on one side of $H_i$, and where $\mathcal S$ is a closed subset that lies on a finite union of compact submanifolds of $\tilde X$ of dimensions $n-2$ or less.  The sets $H_i$ are called boundary hypersurfaces of $X$. We call $\mathcal S$ the singular set (or sometimes corners), and write $\partial X \setminus \mathcal S$ for the regular part of the boundary.
\end{def1}
Throughout with paper we assume that $g$ is a metric on $X$ that can be extended to a smooth metric $\tilde g$ on $ \tilde X$.  A geodesic ball $B_r(x)$ in $X$ (centered at $x \in X$) is defined to be $\tilde B_r(x) \cap \tilde X$ where $\tilde B_r(x)$ is the geodesic ball in $(\tilde X, \tilde g)$ of radius $r$ centered at $x$.   We also define $\text{inj}(X, g) = \inf_{x \in X} \text{inj}(x)$, where $\text{inj}(x)$ is the largest $R$ such that $\tilde B_R(x)$ is embedded in $\tilde X$. Note that this definition of injectivity radius is extrinsic and is smaller than or equal the intrinsic definition one can consider.

\section{Proof of local $L^p$ restriction estimates}
As we discussed, the main ingredient is Lemma \ref{LocalLp} whose proof is similar to Sogge's local $L^p$ estimates \cite{So15} and follows by imitation. We give the proof since later in Section \ref{PSQE} we need to make some modifications of this  proof for manifolds with corners. 
\begin{proof} [\textbf{Proof of Lemma \ref{LocalLp}}] First we choose a nonnegative function $\rho \in C^\infty (\R)$ satisfying
\begin{equation} \rho(0)=1, \quad \text{and } \, \, \supp \hat \rho(t) \subset [-\frac12, \frac12].
\end{equation}
Then we define the operator
\begin{align} A_{\la,r} & = \frac1\pi \int_{-\infty}^\infty r^{-1} \hat \rho(r^{-1}t) \, e^{it \sqrt{\la}}
\cos(t \sqrt{\Delta_g}) \, dt  \label{A} \\ &=\rho(r(\sqrt{\la}-\sqrt{\Delta_g}))+\rho(r(\sqrt{\la}+\sqrt{\Delta_g})) \nonumber .\end{align}
By the properties of $\rho$, we have
$$A_{\la,r} \psi_\la =\big (1+\rho(2r \sqrt{\la})\big ) \, \psi_\la.$$
Hence since $\rho$ is nonnegative we have
$|\psi_\la| \leq |A_{\la,r} \psi_\la|$, which implies that
$$\| \psi_\la \|_{L^{p}(B_{r/2}(x) \cap \Sigma)}\le \| A_{\la,r} \psi_\la \|_{L^{p}(B_{r/2}(x) \cap \Sigma)}.
$$
To prove the lemma, it is enough to show that for all $f \in C^\infty (X)$
$$ \|A_{\la,r}f\|_{L^p(B_{r/2}(x) \cap \Sigma)} \le C_1 r^{-\frac{1}{2}}\la^{\delta}\|f\|_{L^2(B_r(x))}. $$ 
To show this, we first observe that by the finite speed of propagation property of  $\cos \big (t \sqrt{\Delta_g} \big )$, the integral kernel $\cos \big (t \sqrt{\Delta_g}\big )(x,y)$
vanishes if $d_g(x, y) > t$.
Therefore, from the fact that $\supp \hat \rho (t) \subset [-\frac12, \frac12]$, the integral kernel
$A_{\la,r}(x,y)$ of $A_{\la,r}$ satisfies
$$A_{\la,r}(x,y)=0, \quad \text{if } \, \, d_g(x,y)> \frac{r}{2}.	$$
 This in particular shows that $$ \|A_{\la,r}f\|_{L^p(B_{r/2}(x) \cap \Sigma)} = \|A_{\la,r}(f|_{B_r(x)})\|_{L^p(B_{r/2}(x) \cap \Sigma)} \leq  \|A_{\la,r}(f|_{B_r(x)})\|_{L^p( \Sigma)}$$ As a result, our local $L^p$ restriction estimate is reduced to proving the global restriction estimate \footnote{One can probably use a quantitative version of \cite{Ta12} and \cite{HaTa} to give an alternate proof of this, since $A_{\lambda, r}$ is a quasimode.}
$$\|A_{\la,r}f\|_{L^{p}(\Sigma)}\le C_1r^{-\frac12}\la^{\delta}\|f\|_{L^2(X)},$$
for all $r \in [ \lambda^{-1/2}, \frac12 \,\text{inj}(X, g)]$ and $C_1$ that is uniform in $r$ and $\lambda$.

To prove this reduced estimate, we first recall that 
$\Pi_k=\bigoplus_{ \sqrt{\lambda_j} \in\,[\sqrt{k}, \sqrt{k}+1]} \Pi_{E_{\lambda_j}}$, where $\Pi_{E_{\lambda_j}}$ is the orthogonal projection operator onto the eigenspace $E_{\lambda_j}=\, \text{ker}\, (\Delta_g - \lambda_j)$, and that 
$$A_{\la,r}f=\sum_{j=0}^\infty \bigl[\rho(r( \sqrt{\la}-\sqrt{\la_j}))+\rho(r(\sqrt{\la}+\sqrt{\la_j}))\bigr] \, \Pi_{\lambda_j}f.
$$
Because $\rho\in {\mathcal S}(\R)$, we have  for every $N \in \IN$ and $\la\ge1$
$$|\rho(r(\sqrt{\la}-\sqrt{\la_j}))|+|\rho(r(\sqrt{\la}+\sqrt{\la_j}))|\le c_N(1+r|\sqrt{\la}-\sqrt{\la_j}|)^{-N}. $$
Therefore,
\begin{equation} \label{PiEstimate}
	\|\Pi_k A_{\la,r}f\|_{L^2(X)}\le c_N(1+r|\sqrt{\la}-\sqrt{k}|)^{-N}\|\Pi_k f\|_{L^2(X)}, \quad  N \in \IN. \end{equation}
Now, for each $m \in \Z$ let $I_m =\Bigl[\sqrt{\la} +\frac{2m-1}{r}, \sqrt{\la}+ \frac{2m + 1}{r}\Bigr).$
Since the number of intervals $[\sqrt{k}-1,\sqrt{k})$ that intersect $I_m$ as $\sqrt{k}$ varies in $\N$ is bounded by $\frac{2}{r}$, 
we can use the Cauchy-Schwarz inequality to obtain
$$
\Bigl\|\, \sum_{\{\sqrt{k} \in \N: \, [\sqrt{k}-1,\sqrt{k})\cap I_m \ne \emptyset\}}\Pi_k g
\Bigr\|_{L^{p}(\Sigma)}
\le 2 r^{-\frac12}\Bigl( \, \sum_{\{ \sqrt{k} \in \IN: \, [ \sqrt{k}-1, \sqrt{k})\cap I_m \ne \emptyset\}} \|\Pi_k g\|_{L^{p}(\Sigma)}^2
\, \Bigr)^{\frac {1}{2}}.
$$
By using this inequality for $g= A_{\lambda, r}f$, using our assumption (\ref{NormOfSpectralCluster}) that 
$$||\Pi_k ||_{L^2(X) \to L^p(\Sigma)} \leq C k^{\delta},$$ and using (\ref{PiEstimate}), we see that 
\begin{align*}
    \Bigl\|\, \sum_{\{ \sqrt{k} \in \IN: \, [\sqrt{k}-1, \sqrt{k})\cap I_m \ne \emptyset\}} & \Pi_k A_{\la,r}f  \Bigr\|_{L^{p}(\Sigma)}
	\\
	&\le 2C r^{-\frac12}\Bigl( \, \sum_{\{ \sqrt{k} \in \IN: \, [\sqrt{k}-1, \sqrt{k})\cap I_m \ne \emptyset\}} k^{2\delta} \|\Pi_k
	A_{\la,r} f \|_{L^2(X)}^2\, \Bigr)^{\frac12} \notag 
	\\
	&\le 2C c_N r^{-\frac12} \Big (1+|\sqrt{\la} +\frac{2m+1}{r}| \Big )^{2\delta}(1+|m|)^{-N}\| f \|_{L^2(X)}.
\end{align*}
Finally, by the above inequality 
\begin{align*}
   \|A_{\la,r}f\|_{L^{p}(\Sigma)} &\le \sum_{m \in \Z}
	\bigl\|A_{\la,r}(\sum_{\la_j\in I_m}\Pi_{E_{\lambda_j}}f)\bigr\|_{L^{p}( \Sigma )}
	\\
	&=\sum_{m\in Z}	\bigl\|\sum_{ \sqrt{k} \in \IN :[\sqrt{k}-1, \sqrt{k})\cap I_m \ne \emptyset} \Pi_k A_{\la,r}(\sum_{\la_j\in I_m} \Pi_{E_{\lambda_j}} f)\bigr\|_{L^{p}(\Sigma)}
	\\
	&\le 2C c_Nr^{-\frac12}\sum_{m \in \Z}(1+|m|)^{-N}\, \Big (1+| \sqrt{\la} + \frac{2m +1}{r}|\bigr)^{2\delta }\|f\|_{L^2(X)}
	\\
	&\le C_1 r^{-\frac12}\la^{\delta}\|f\|_{L^2(X)},
\end{align*}
where we have used
$$| \sqrt{\la} +r^{-1}(2m+1)|  \leq 2 \sqrt{\la} (1+|m|), $$ which is implied from the assumption $r \geq \la^{-1/2}$. 
\end{proof}

\section{Proofs of global $L^p$ restrictions estimates}  
In this section we prove Theorem \ref{LpEstimatesSS} (and Theorems \ref{LpQENC} and \ref{LpQE} as its corollaries), using Lemma \ref{LocalLp}. In the course of the proof we also show the following corollary which gives global $L^p$ estimates for restrictions of eigenfunctions to $\Sigma$ in terms of local $L^2$ norms on small balls in $(X,g)$ centered on $\Sigma$, and $L^2$ norm on a small tube around $\Sigma$.  

\begin{cor} \label{GlobalLp} Under the assumptions of Lemma \ref{LocalLp}, there exists $r^*$ dependent on $(X, g)$ and $\Sigma$, such that uniformly for all $r \in [\lambda_j^{-\frac12}, r^*]$:
\begin{equation} \label{GlobalLpEstimates} 
||\psi_{\lambda_j} ||_{L^p(\Sigma)} \leq C_3 r^{-\frac{1}{2}} \left ( \sup_{x \in \Sigma} || \psi_{\lambda_j}||_{L^2(B_{r}(x))} \right)^{\frac{p-2}{p}} \left ( || \psi_{\lambda_j}||_{L^2(\mathcal T_{r}(\Sigma))} \right )^{\frac{2}{p}} \lambda_j^\delta, 
\end{equation}
where $\mathcal T_{r} (\Sigma)= \{ p \in X; d_g(p, \Sigma) < r \}$ is the tube of radius $r$ around $\Sigma$. 
\end{cor}

\begin{proof}[\textbf{Proofs of Theorem \ref{LpEstimatesSS} and Corollary \ref{GlobalLp}}] We approach as \cite{HeRi} and \cite{So15}. We fix $r < r_0/2=\text{inj}(X, g)/2$ and we choose a set of points $\mathcal J =\{ x_i \} \subset \Sigma$ such that 
$$\Sigma \subset \bigcup_{x_i \in \mathcal J} B_{r/2}(x_i),$$
in such a way that any point $p$ in $X$ belongs to at most $c$ many (or none) of the double balls $B_{r}(x_i)$, where $c$ only depends on $(X, g)$. We then select $r^* < \frac{r_0}{2}$ small enough so that for all $r \leq r^*$
$$ \text{Vol}_g (\mathcal T_r (\Sigma)) \leq c_0 r^{n-k}, $$ for some $c_0$ independent of $r$. This is possible because $\Sigma$ is a compact embedded smooth submanifold of $X \backslash \partial X$ or of $\partial X \backslash \mathcal S$ of dimension $k$. We also realize that by making $r^*$ even smaller we can make sure that  for all $x \in \Sigma$  $$\text{Vol}_g (B_{r/2}(x)) \geq a_1 r^{n}$$ for some uniform $a_1$. This is obvious if $\Sigma$ is an interior submanifold. However, if $\Sigma \subset \partial X \backslash \mathcal S$, for a given $x \in \Sigma$ we first change our coordinates near $x$ to the upper half plane, then we realize that by choosing $r^*$ sufficiently small, for any $r < r^*$ we can fit a half Euclidean ball of radius $\frac{r}{3}$ inside the geodesic ball $B_{r/2}(x)$. We then have 
$$ a_1 \, \text{card}( \mathcal J) r^n \leq \sum_{x_i \in \mathcal J} \text{Vol}_g(B_{r/2}(x_i)) \leq c \text{Vol}_g( \mathcal T_r(\Sigma)) \leq c c_0 r^{n-k} \, .$$ 
Therefore, we must have $ \text{card}( \mathcal J) \leq B r^{-k}$ for some $B$ that is independent of $r$.

Then using Lemma \ref{LocalLp} we write 
\begin{align*}
\| \psi _\la \|_{L^{p}(\Sigma)}^{p}
&\le \sum_{x_i \in \mathcal J} \| \psi_\la\|_{L^{p}(B_{r/2}(x_i) \cap \Sigma)}^{p}
\\
&\le C_1^p \la^{p \delta} \,  r^{-\frac{p}{2}}\sum_{x_i \in \mathcal J }\| \psi_\la\|_{L^2(B_{r}(x_i))}^{p}
\\
&\le C_1^p \la^{\delta p}  \,  r^{-\frac{p}{2}}\,  \Bigl(\sup_{x_i  \in \mathcal J}
\| \psi _\la\|_{L^2(B_{r}(x_i))}^{p-2}\Bigr)\sum_{x_i \in \mathcal J}
\| \psi_\la\|_{L^2(B_{r}(x_i))}^2
\end{align*}
Corollary \ref{GlobalLp} follows because $$\sum_{x_i \in \mathcal J}
\| \psi_\la\|_{L^2(B_{r}(x_i))}^2 \leq c \| \psi_ \la \| ^2 _{\mathcal T _r (\Sigma)}.$$
Theorem \ref{LpEstimatesSS} follows by observing that under the $L^2$ assumption of this theorem and also our observation that $\text{card}(\mathcal J)< Br^{1-n}$, we have
\begin{align*} \| \psi _\la \|_{L^{p}( \Sigma )}^{p}
&\le C_1^p \la^{p \delta} \,  r^{-\frac{p}{2}} \sum_{x_i \in \mathcal J }\| \psi_\la\|_{L^2(B_{r}(x_i))}^{p} \\  &\leq C_2^p r^{-k} r^{\frac{p}{2}(n-1)} \lambda^{p \delta}. \end{align*}
\end{proof}

\begin{proof} [\textbf{Proof of Theorem \ref{LpQENC}}]
As discussed in the introduction, this theorem is resulted immediately from Theorem \ref{LpEstimatesSS} combined with the results of \cite{BuGeTz, Hu, ChSo} on the spectral cluster operators, and also the small scale QE results  (\ref{QENC}) of \cite{HeRi}.
\end{proof}

\begin{proof} [\textbf{Proof of Theorem \ref{LpQE}}]
In the boundaryless case $\partial X =\emptyset$, this theorem follows easily from Theorem \ref{LpEstimatesSS} and the below lemma combined with the quantum ergodicity results of \cite{Sh, CdV, Ze87}. We will prove this theorem for the case of  $\partial X \neq \emptyset$ in Section \ref{Boundary}. 
\end{proof}

\begin{lem} \label{QElemma} Let $(X,g)$ be a compact connected Riemannian manifold of dimension $n$, with or without boundary. When $\partial X \neq \emptyset$, we assume that $\partial X$ is piecewise smooth. Let $\{ \psi_{\la_j} \}_{j \in S}$ be a  sequence of eigenfunctions of $\Delta_g$ with eigenvalues $\{\lambda_j\}_{j \in S}$ such that for some $r^*< \text{inj}(X,g)$ and for all $r \in (0, r^*)$ and all $x \in X$
\begin{equation} \label{QEonX2}
\int_{B_r(x)} |\psi_{\la_j}|^2 \to \frac{\text{Vol}_g(B_r(x))}{\text{Vol}_g(X)}, \qquad  \lambda_j \xrightarrow{j \in S} \infty.
\end{equation}
Then there exists $r_0(g)$ such that for each $r \in (0, r_0(g))$ there exists $\Lambda_r$ such that for $ \lambda_j \geq \Lambda_r$ we have
$$ \int _{B_{r}(x)} | \psi_{\la_j}|^2 \leq K r^n, $$ uniformly for all $x \in X$. Here, $K$ is independent of $r$, $j$, and $x$. 
 \end{lem}
We point out that this lemma is obvious when $x$ is fixed, however to have it work uniformly for all $x \in X$ we need to use a covering argument as follows. 

\begin{proof}First we choose $r_0(g) < \frac{r^*}{2}$ small enough so that for all $ r <r_0(g)$
$$  \text{Vol}(B_{2r}(x))  \leq a r^n,$$ for some positive $a$ that is independent of $r$ and $x$. We cover $(X, g)$ using geodesic balls $\{B_{r/2}(x_i) \}_ {x_i \in \mathcal I}$ such that $\text{card}\, (\mathcal I)$ is at most $C_0r^{-n}$, where $C_0$ depends only on $(X, g)$. For the existence of such a covering see for instance Lemma 2 of \cite{CM}.  Next for each $x_i \in \mathcal I$, by using (\ref{QEonX2}) for balls $B_{2r}(x_i)$, we can find $\Lambda_{i, r}$ large enough so that for $\lambda_j \geq \Lambda_{i, r}$
$$\int _{B_{2r}(x_i)} | \psi_{\la_j}|^2 \leq K r^n, $$
with $K= \frac{2a}{\text{Vol}(X)}$. We claim that $ \Lambda_r =\max_{i \in \mathcal I} \{ \Lambda_{i, r} \}$ would do the job for all $x$ in $X$. So let $x$ be in $X$ and $r$  be as above. Then $x \in B_{r/2}(x_i) $ for some $i \in \mathcal I$ and clearly one has $B_r(x) \subset B_{2r}(x_i) $.  This and the above inequality prove the lemma. 
\end{proof}

\section{Improved supnorms and the number of nodal domains}\label{PSQE}
In this section we prove Theorems \ref{LinftyQE} and \ref{NN}. The main technical obstacle is the presence of singular points on the boundary, which as we show can be overcome by the finite speed of propagation property of the solutions to the wave equation. Let us start with the proof of improved $L^\infty$ estimates.

\subsection{\textbf{Proof of Theorem \ref{LinftyQE} }} First let us recall the following two results of \cite{So02} and \cite{Xu} on the supnorm of spectral clusters.
\begin{theo}[Sogge \cite{So02} and Xu \cite{Xu}] \label{LinftyPi} Let $(X,g)$ be a compact connected smooth Riemannian manifold  of dimension $n$ with smooth boundary, and $\Pi_\la$ be the spectral cluster operator associated to $\Delta_g$ with Dirichlet or Neumann boundary conditions.  Then
$$ \sup_X | \Pi_\la (f) | \leq C \la^{\frac{n-1}{4}} || f  ||_{L^2(X)},$$
$$ \sup_X | \la^{-\frac12} \nabla \, \Pi_\la (f) | \leq C \la^{\frac{n-1}{4}} || f ||_{L^2(X)}, $$
where the constant $C$ depends only on $(X, g)$. 
\end{theo}

Now suppose $(X, g)$ is a manifold with piecewise smooth boundary, isomorphically embedded in a compact Riemannian manifold $(\tilde X, g)$ without boundary of the same dimension (see definition \ref{PS}). For  each $s \in [0,  \text{inj}(X, g)/2]$ we choose a family $\{(X_s, g)\}$ of compact submanifolds of $( \tilde X, g)$, with smooth boundary for $s>0$,  such that
\begin{equation} \label{Xs} \begin{cases} X_0=X, \quad X \subset X_s, \\
$$ X \backslash \mathcal T_s( \mathcal S) = X_s \backslash \mathcal T_s( \mathcal S), \end{cases} \end{equation} where $\mathcal T_s( \mathcal S) =\{ x \in \tilde X; \, d(x, \mathcal S) < s \}$. This is possible by smoothing out $X$ in $\tilde X$ near its corners. We note that since the metrics on $X$ and $\{X_s\}$ are the restrictions of the metric $g$ on $\tilde X$, we have used the same notation for all of them. Let $\Delta_s$ be the Laplace-Beltrami operator on $(X_s, g)$ with Dirichlet or Neumann boundary conditions and consider the family of operators $A_{\la, r}(s)$ in (\ref{A}) defined by
$$A_{\la,r}(s)  = \frac1\pi \int_{-\infty}^\infty r^{-1} \hat \rho(r^{-1}t) \, e^{it \sqrt{\la}}
\cos(t \sqrt{\Delta_s}) \, dt, $$
where $\rho \in C^ \infty(\R)$, $\rho(0)=1$, and $\supp \hat \rho(t) \subset [ -\frac12, \frac12]$.  Now suppose $\psi_\la$ is an eigenfunction of $\Delta_0$. Then by repeating the argument of the proof of Lemma \ref{LocalLp}, we have 
$$A_{\la,r}(0) \psi_\la =\big (1+\rho(2r \sqrt{\la})\big ) \, \psi_\la, $$
and because $\rho \geq 0 $ we must have
\begin{equation}\label{above3} \| \psi_\la \|_{L^{\infty}(B_{r/2}(x))}\le \| A_{\la,r}(0) \psi_\la \|_{L^{\infty}(B_{r/2}(x))}. \end{equation} Since the finite speed of propagation of $\cos \big (t \sqrt{\Delta_0} \big )$ also holds on manifolds with piecewise smooth boundary, $\cos \big (t \sqrt{\Delta_0}\big )(x,y)$
vanishes if $d_g(x, y) > |t|$. Hence because  $\hat \rho (t) =0$ for $|t| \geq \frac{1}{2}$, we have
$$A_{\la,r}(0)(x,y)=0, \quad \text{if } \, \, d_g(x,y)> \frac{r}{2}. $$
which implies that for all $ f \in C^\infty(X)$ 
\begin{equation}\label{above4} \|A_{\la,r}(0)f\|_{L^\infty(B_{r/2}(x))} = \|A_{\la,r}(0)(f|_{B_r(x)})\|_{L^\infty(B_{r/2}(x))}. \end{equation}  Next we choose $s>0$ small enough so that $ X \backslash \mathcal T_{4s}(\mathcal S)$ has a non-empty interior. We claim that for $x \in X$ and $y \in X \backslash \mathcal T_{2s}(\mathcal S)$ as long as $d(x, y) \leq  |t| < \frac{s}{2}$, we have 
\begin{equation} \label{Claim} \cos \big (t \sqrt{\Delta_0} \big )(x, y) = \cos \big (t \sqrt{\Delta_s} \big )(x, y). \end{equation} The proof of this, which we are about to present, is basically the same as the proof of finite speed of propagation. To do this, let $h$ be an arbitrary real-valued smooth function supported in $X \backslash \mathcal T_{2s}$ and define
$$ u(t, x) =   \left ( \cos \big (t \sqrt{\Delta_0} \big ) - \cos \big (t \sqrt{\Delta_s} \big ) \right) (h).$$ By the properties of $X_s$, it is clear that $u$ satisfies
$$ \begin{cases} \partial_t^2 u(t, x) = -\Delta_0 u(t, x), \quad t \in \IR, \; x \in X, \\
u(0, x) =0,  \quad  \; x \in X,\\
\partial_t u(0, x) =0, \quad  \; x \in X,  \\
\forall t \in \IR, \; \forall x \in \partial X \backslash \mathcal T_{s} (\mathcal S): \; \begin{cases} u(t, x) =0 \; \text{(Dirichlet case)}, \\  \partial_n u(t, x) =0\; \text{(Neumman case)}. \end{cases} \end{cases} $$ Then for $t >0$  and $y$ in the support of $h$, consider the local energy $$ E(t) = \int_{B_{s-t}(y)\cap X} |\partial_t u(t, x)|^2 + | \nabla_g u(t, x) |^2  \; d_gv,$$ where geodesic balls are defined in $\tilde X$.  Differentiating $E(t)$, and using the coarea formula and the divergence theorem, we obtain
\begin{align*} 
E'(t) & = - \int_{S_{s-t}(y)\cap X} |\partial_t u|^2  + | \nabla_g u|^2 \; d_g \sigma
\\ & \quad  + 2 \int_{B_{s-t}(y)\cap X} (\partial_t^2 u)( \partial_t u) + < \nabla_g \partial_t u, \nabla_g u>  \; d_gv\\  &= - \int_{S_{s-t}(y) \cap X} |\partial_t u|^2  + | \nabla_g u |^2 \; d_g \sigma + 2 \int_{\partial \big(B_{s-t}(y) \cap X \big )} (\partial_t u)(\partial_n u) \; d_g \sigma.
\end{align*}
We then note that $\partial \big(B_{s-t}(y) \cap X \big ) = (S_{s-t}(y) \cap X ) \cup (\partial X \cap B_{s-t}(y) )$. Since $y \in X \backslash \mathcal T_{2s}(\mathcal S)$, the set $\partial X \cap B_{s-t}(y)$ never intersects $ \mathcal T_s ( \mathcal S)$ and therefore we can use the boundary conditions to see that with either Dirichlet or Neumann boundary conditions, $ \partial_t u\, \partial_n u$ vanishes on  $ \partial X \cap B_{s-t}(y)$. As a result, 
$$ E'(t) = - \int_{S_{s-t}(y) \cap X} \left ( |\partial_t u|^2  + | \nabla_g u |^2  - 2(\partial_t u)(\partial_n u) \right ) \, d_g \sigma \leq 0 \; .$$ Since $E(0)=0$, we must have $E(t)=0$, which implies that $u(t, x)=0$ for all $t$, $x \in X $, and $y  \in X  \backslash \mathcal T_{2s}(\mathcal S)$ with $d(x, y) < s -|t|$. Since $h$ is arbitrary, if we choose $|t| < s/2$, the claim (\ref{Claim}) follows. Consequently if $r <s/4$, we have
$$ A_{\la, r}(0)(x, y) = A_{\la, r}(s)(x, y), \quad x \in X, \, y  \in X  \backslash \mathcal T_{2s}(\mathcal S), \; d(x, y) < 2r,$$ 
$$ \nabla_x A_{\la, r}(0)(x, y) = \nabla _x A_{\la, r}(s)(x, y), \quad x \in X, \, y  \in X  \backslash \mathcal T_{2s}(\mathcal S), \; d(x, y) < 2r.$$This identity tells us that if we choose any $x \in X \backslash \mathcal T_{3s}(\mathcal S)$, any $r < s/4$, and any $f \in C^\infty(\tilde X)$ we have
\begin{equation} \label{above1} \|A_{\la,r}(0)(f|_{B_r(x)})\|_{L^\infty(B_{r/2}(x))}=\|A_{\la,r}(s)(f|_{B_r(x)})\|_{L^\infty(B_{r/2}(x))}, \end{equation}
\begin{equation} \label{above2}  \| \nabla \left ( A_{\la,r}(0)(f|_{B_r(x)}) \right )\|_{L^\infty(B_{r/2}(x))}=\| \nabla \left ( A_{\la,r}(s)(f|_{B_r(x)})\right )\|_{L^\infty(B_{r/2}(x))}. \end{equation} However, since $X_s$ has a smooth boundary we can use Theorem \ref{LinftyPi} and plug it in the argument of \cite{So15}, or in the proof of our Lemma \ref{LocalLp}, to obtain
$$ \|A_{\la,r}(s)f\|_{L^ \infty (X_s)} \le C_s r^{-\frac{1}{2}}\la^{\frac{n-1}{4}}\|f\|_{L^2(X_s)}, $$
$$ \| \la^{-\frac12}\nabla A_{\la,r}(s)f\|_{L^ \infty (X_s)} \le C_s r^{-\frac{1}{2}}\la^{\frac{n-1}{4}}\|f\|_{L^2( X_s)}, $$
for all $f \in L^2(X_s)$. If we use these inequalities with $f$ replaced by $f|_{B_{r/2}(x)}$ where $x  \in X \backslash \mathcal T_{3s}(\mathcal S)$, then using (\ref{above3}), (\ref{above4}), (\ref{above1}), and (\ref{above2}) we arrive at
$$||\psi_{\lambda} ||_{L^\infty(B_{r/2}(x))} \leq C_s r^{-\frac{1}{2}} || \psi_{\lambda}||_{L^2(B_{r}(x))} \lambda^\frac{n-1}{4},  $$
$$ || \la^{-\frac12} \nabla \psi_{\lambda} ||_{L^\infty(B_{r/2}(x))} \leq C_s r^{-\frac{1}{2}} || \psi_{\lambda}||_{L^2(B_{r}(x))} \lambda^\frac{n-1}{4} .  $$ We must emphasize that these estimates are for the eigenfunctions of $\Delta_0$ (with Dirichlet or Neumann boundary conditions) on $(X, g)$, which has a non-smooth boundary, and are only valid for $s<s_0$ (where $s_0$ depends only on $(X, g)$), $x  \in X \backslash \mathcal T_{3s}(\mathcal S)$, and $\la^{-\frac12} \leq r <s/4$. The constant $C_s$ depends on $s$ and is independent of $r$, $\la$, and $x$.

Now let us assume the billiard flow on $(X, g)$ is ergodic. Then by the QE theorem of \cite{ZZ}, for any ONB $\{ \psi_{\la_j} \}_{ j \in \mathbb N}$ of eigenfunctions of $\Delta_0$ with Dirichlet or Neumann boundary conditions, there exists $S \subset \IN$ of full density such that (\ref{QEonX2}) holds. Hence by Lemma \ref{QElemma}, there exists $r_0$ dependent only on $(X, g)$ such that for each $r \in (0, r_0(g))$ there exists $\Lambda_r$ such that for $ \lambda_j \geq \Lambda_r$  ($j \in S$)
$$ \int _{B_{r}(x)} | \psi_{\la_j}|^2 \leq K r^n, $$ where $K$ is independent of $r$, $j$, and $x$. Applying these local $L^2$ estimates to the local $L^\infty$ estimates given above, we get for $j \in S$,  $x  \in X \backslash \mathcal T_{3s}(\mathcal S)$, and $\la_j^{-\frac12} \leq r <s/4$:
\begin{align} ||\psi_{\lambda_j} ||_{L^\infty(B_{r/2}(x))} \leq K^{\frac12}C_s r^{\frac{n-1}{2}}  \lambda_j^\frac{n-1}{4}, \label{local1}   \\ \label {local2} || \la_j^{-\frac12} \nabla \psi_{\lambda_j} ||_{L^\infty(B_{r/2}(x))} \leq K^\frac12 C_s r^{\frac{n-1}{2}} \lambda_j^\frac{n-1}{4} .  \end{align}  Without the loss of generality we assume that $C_s$ and $\Lambda_r$ are continuous monotonic functions with $C_s \to \infty$ as $s \to 0$ and $\Lambda_r \to  \infty$ as $r \to 0$. We also assume that $\Lambda_r \geq r^{-2}$. Then we choose $r(s)$ a continuous monotonic function of $s$ on $(0, s_0)$ such that $r(s_0) < r_0$, $r(s) < s/4$ and
$$ r(s)^{\frac{n-1}{2}} C_s \to 0, \quad \text{as} \quad s \to 0.$$  We note that there exists $j_0 \in S$ such that for all $j >  j_0$ we have $ \la_j \geq \Lambda_{r(s_0)}$. We discard all $ j \leq j_0$ from the set $S$, noting that the remaining set is still of full density (we call it $S$ again). Then for each $j \in S$ we define $s_j$ by $\Lambda_{r(s_j)}= \la_j$.  By the choice of $s_j$ and $\Lambda_r$, the conditions $ \la_j \geq \Lambda_{r(s_j)}$ and $ \la_j^{-1/2} \leq r(s_j)$ are automatically satisfied, and  Theorem \ref{LinftyQE} follows by (\ref{local1}) and (\ref{local2}) if we choose $\ep_j = 3 s_j$.   

\subsection{ \textbf{Proof of Theorem \ref{NN}; Number of nodal domains}}
We refer the reader to the paper \cite{ZeJJ} of Jung- Zelditch for the complete details of the facts we use in this section. Here we only present the necessary modifications that are needed to extend their results to all ergodic billiards with piecewise smooth boundary. Let us recall from \cite{ZeJJ} the following facts. 

\begin{prop}\label{JJZ} Let $(X, g)$ be a compact connected smooth Riemannian manifold of dimension $n=2$ with smooth boundary, parametrized with arc length. Let $\{ \psi_{\la_j} \}_{j \in S}$ be a sequence of eigenfunctions of $\Delta_g$ with Dirichlet or Neumann boundary conditions with eigenvalues $\{ \la_j \}_{j \in S}$, and $\{\psi^b_{\la_j}(s)\}_{j \in S}$ be its associated Cauchy data on $\partial X$  defined by 
\begin{equation} \label{traces} \psi^b_{\la_j}(s)= \begin{cases} \psi_{\la_j}(s), \qquad \quad \quad \; \; \text{in the Neumann case},\\ \la_j^{-1/2} \partial_n \psi_{\la_j} (s), \quad \text{in the Dirichlet case} .\end{cases} \end{equation} Let $ \Sigma \subset \partial \Omega$ be a closed connected arc.  In addition, assume
\begin{itemize}
\item[(a)] For all non-negative $f \in C^\infty_0( \partial X)$, with $\text{supp}(f) \subset \Sigma$, there exist $\Lambda_{f}$ and $c_1>0$ independent of $\Sigma$ and $\la_j$,  such that $$ \int_ \Sigma f(s) | \psi_{\la_j}(s) |^2 \, ds  \geq c_1 \int_ \Sigma f(s) \, ds, \quad   \la_j \geq \Lambda_{f}. $$
\item[(b)] For all $f \in C^\infty_0( \partial X)$, with $\text{supp}(f) \subset \Sigma$, there exist $\Lambda'_{f}$ and $c_2>0$ independent of $\Sigma$ and $\la_j$, such that
$$ \left |\int_ \Sigma f(s) \psi_{\la_j}(s) \, ds \right |^2 \leq c_2 \la_j^{-\frac12} \int_ \Sigma |f(s)|^2 \, ds,  \quad   \la_j \geq \Lambda'_{f}. $$
\item[(c)] For $j \in S$ { $$\sup_\Sigma |\psi^b_{\la_j}| = o( \la_j^{\frac14}). $$}
\end{itemize}
Then for $j \in S$ $$ \sharp \{s \in \Sigma; \; \psi^b_{\la_j}(s)=0 \} \to \infty \quad \text{as} \quad \la_j \to \infty. $$ 
\end{prop}
\begin{proof} See pages 818-819 of \cite{ZeJJ}. \end{proof}
Also in \cite{ZeJJ}, using a a topological argument based on \cite{GRS},  the authors proved that:
\begin{theo}\label{Topological} Under the initial assumptions of the above proposition (but not the additional assumptions (a)-(c)), one has
$$\text{the number of nodal domains of} \; \psi_{\la_j} \; \geq  \; \frac12 \sharp \{s \in \partial X; \; \psi^b_{\la_j}(s)=0 \} - c(X,g),$$
where $c(X,g)$ depends only on $(X,g)$. 
\end{theo}
\begin{proof} See Theorem 6.3 of \cite{ZeJJ}. Although this result is only proved for surfaces with smooth boundary, it still follows even if $\partial X$ is piecewise smooth.  The only difference is that one needs to assign a vertex to each singular point on the boundary, which can change the constant $c(X, g)$. \end{proof}

Let us now discuss the additional assumptions in Proposition \ref{JJZ}. Assumption (a) holds in general for piecewise smooth ergodic manifolds by a result of Burq \cite{Bu} which shows a quantum ergodicity property for the boundary values $\psi^b_{\la_j}(s)$ (see also \cite{HaZe} where this is proved for Euclidean ergodic billiards). Hence given an ONB of eigenfunctions of $\Delta_g$ on such a manifold, we can find a full density subsequence  for which (a) holds. Condition (c), with $\Sigma= \partial X$, was proved in \cite{SoZe14} for non-positively curved manifolds with smooth concave boundary,  but in the more general case of ergodic manifolds with piecewise smooth boundary we can use our Theorem \ref{LinftyQE}. In the case of manifolds with smooth boundary, condition (b) follows from:
\begin{theo} [\cite{HHHZ}] \label{HHHZ} Let $(X, g)$ be a compact connected smooth Riemannian manifold of dimension $n \geq 2$ with smooth boundary. Let $\{ \psi_{\la_j} \}_{j \in \IN}$ be an ONB of eigenfunctions of $\Delta_g$ with Dirichlet or Neumann boundary conditions, and let $\{ \psi^b_{\la_j}(s) \}_{ j \in \IN}$ be its boundary traces defined in (\ref{traces}). Suppose $\rho(t) \in C^\infty (\R)$ is a function satisfying
$$ \rho(t) \geq 0,  \quad  \rho(0)=1, \quad \supp \hat \rho(t) \subset [-\frac12, \frac 12].$$ Then for $r_0$ sufficiently small and for all $ f \in C^\infty(\partial X)$
\begin{equation} \label{completenessN}   \frac{\pi}{2} \sum_{ j=1}^\infty \rho_0 \big (\sqrt{\la} - \sqrt{\la_j} \big ) \langle f, \psi^b_{\la_j} \rangle \psi^b_{\la_j}(x) = f(x) + o(1), \end{equation} where $\rho_0$ is a function whose Fourier transform is $ \hat \rho (r_0^{-1} t)$. Also $\langle \, ,  \, \rangle$ is the natural inner product in $L^2(\partial X)$, and the above convergence is in $L^2(\partial X)$. 
\end{theo}

\begin{rema} In fact in the Dirichlet case it is more convenient for us to use the following version
\begin{equation} \label{completenessD} \frac{\pi}{2} \sum_{ j=1}^\infty \rho_0 \big (\sqrt{\la} - \sqrt{\la_j} \big ) \langle f, \partial_n \psi_{\la_j} \rangle \partial_n \psi_{\la_j}(x) = f(x) \la + o(\la). \end{equation} We prefer this version because its related to the operator $\cos (t \sqrt{ \Delta_g})$ which satisfies the finite speed of propagation, as opposed to the operator $\frac{\cos (t \sqrt{ \Delta_g})}{ \sqrt{ \Delta_g}}$ involved in (\ref{completenessN}). 
\end{rema}

Before showing that Theorem \ref{HHHZ} also works for manifolds with piecewise smooth boundary for $f$ with $\supp(f) \subset \Sigma \subset \partial X \backslash \mathcal S$, let us make the observation that since for all $N \in \IN$, $$\rho_0 \big (\sqrt{\la} + \sqrt{\la_j} \big )= O_N (\la^{-N} \la_j^{-N}),$$ and since 
$$\frac1\pi \int_{-\infty}^\infty  \hat \rho(r_0^{-1}t) \, e^{it \sqrt{\la}}
\cos(t \sqrt{\Delta_g}) \, dt  =\rho_0(\sqrt{\la}-\sqrt{\Delta_g})+\rho_0(\sqrt{\la}+\sqrt{\Delta_g}), $$
(\ref{completenessN}) can be rewritten in the Neumann case as
\begin{equation} \label{N} \;  \frac12 \int_{-\infty}^\infty \int_{\partial X}  \hat \rho(r_0^{-1}t) \, e^{it \sqrt{\la}}
\cos(t \sqrt{\Delta_g})(x, y) f(y) \, d_g\sigma \, dt = f(x) + o(1), \end{equation} and in the Dirichlet case  (\ref{completenessD}) can be written as
\begin{equation} \label{D} \frac12 \int_{-\infty}^\infty \int_{\partial X}  \hat \rho(r_0^{-1}t) \, e^{it \sqrt{\la}}  \left (\partial_{n_x} \partial_{n_y} \cos(t \sqrt{\Delta_g})(x, y) \right ) f(y) \, d_g\sigma \, dt = f(x) \la + o(\la). \end{equation}
Now suppose $(X, g)$ is a manifold with piecewise smooth boundary and $\Sigma \subset \partial X \backslash \mathcal S$ is a compact submanifold of the regular part of the boundary.  Define $s = d_g(\Sigma, \mathcal S) / 4$ and let $(X_s, g)$ be the manifold with smooth boundary defined in (\ref{Xs}). We saw in (\ref{Claim}) that for $x \in X$ and $y \in X \backslash \mathcal T_{2s}(\mathcal S)$ as long as $d(x, y) \leq  |t| < \frac{s}{2}$, we have 
$$ \cos \big (t \sqrt{\Delta_g} \big )(x, y) = \cos \big (t \sqrt{\Delta_{g, s}} \big )(x, y). $$  Hence if we choose $r_0 <s$ we obtain (\ref{N}) and (\ref{D}), and therefore \footnote{See Remark \footnote{finalremark}.}
(\ref{completenessN}) and (\ref{completenessD}) for $(X, g)$ with non-smooth boundary  as long as $\supp f \subset \Sigma$ and $x \in \Sigma$. 

Finally, let us discuss how condition (b) can be obtained from these. Assume that $\rho$ satisfies the additional assumption $\rho(t) \geq \frac12$ on $[-\alpha ,\alpha]$ for some $\alpha >0$. Then using this property of $\rho$ and its non-negativity, after taking the inner product of both sides of (\ref{completenessN}) and (\ref{completenessD}) with $f \in C^\infty(\partial X)$ satisfying $\supp(f)  \subset \Sigma$, we get
$$\text{Neumann case:} \quad \sum_{|\sqrt{\la} - \sqrt{\la_j}| \leq \alpha/r_0  } | \langle f, \psi_{\la_j} \rangle |^2  \leq  \frac{4}{\pi} \int_{\partial X} |f(x)|^2 + o_f(1) $$
$$\text{Dirichlet case:} \quad  \sum_{|\sqrt{\la} - \sqrt{\la_j}| \leq \alpha/r_0  } | \langle f, \partial_n \psi_{\la_j} \rangle |^2  \leq  \frac{4 \la}{\pi} \int_{\partial X} |f(x)|^2 + o_f(\la).$$
By applying an extraction procedure to the the above estimates, as performed in \cite{ZeJJ}, one can see that for every $\tau \in (0, 1)$ there exists a subsequence of $\{ \psi_{\la_j} \}_{ j \in \N}$ of density $1 -\tau$ for which condition (b) of Proposition \ref{JJZ} holds. Since, as we discussed, all other conditions are satisfied for a full density subsequence, Theorem \ref{NN} follows by letting $\tau \to 0$. 

\begin{rema}\label{finalremark} In the above argument, when we were going back from (\ref{N})  and (\ref{D}) to (\ref{completenessN}) and (\ref{completenessD}) in the case of non-smooth boundary, we implicitly used the fact that if $\supp(f) \subset \Sigma$ then
$$ \| \sum_{ j=1}^\infty \rho_0 \big (\sqrt{\la} + \sqrt{\la_j} \big ) \langle f, \psi^b_{\la_j} \rangle \psi^b_{\la_j}(x) \|_{L^2(\Sigma)} = O (\la^{- \infty}).$$ In order to prove this we need to know that \{$\psi^b_{\la_j}$ \} is polynomially bounded on $\Sigma$. In fact if in the proof of Theorem \ref{LinftyQE} we choose $r=\frac{s}{5}$, we get for $s>0$
$$ \sup_{ X \backslash \mathcal T_s( \mathcal S)}  | \psi_{\la_j}| \leq C'_s \la^{\frac{n-1}{4}} \quad \text{and}  \quad \sup_{ X \backslash \mathcal T_s( \mathcal S)} | \la_j^{-\frac12} \nabla \psi_{\la_j} | \leq C'_s \la^{\frac{n-1}{4}}. $$  
\end{rema}

\section{$L^p$ restrictions on manifolds with boundary} \label{Boundary}
In this section we give a proof of Corollary \ref{BoundaryCor} and also for the remaining part of Theorem \ref{LpQE} when $\partial X \neq \emptyset$.  Let $\Sigma \subset X \backslash \partial X$ be a compact submanifold and define $s = d ( \Sigma, \partial X) /4$. A  proof similar to that of Theorem \ref{LinftyQE} shows that for $x \in X$ and $y \in X \backslash \mathcal T_{2s}(\mathcal \partial X)$ as long as $d(x, y) \leq  |t| < \frac{s}{2}$, we have 
$$ \cos \big (t \sqrt{\Delta_g} \big )(x, y) = \cos \big (t \sqrt{\Delta_{\tilde g}} \big )(x, y), $$ where $(\tilde X, \tilde g)$ is a compact Riemannian manifold without boundary that contains $(X, g)$ as the closure of an open connected subset (see definition (\ref{PS2})).  By imitating the proof of Theorem \ref{LinftyQE}, we get that for $x  \in X \backslash \mathcal T_{3s}(\partial X)$, and for $\psi_\la$  any eigenfunction of $\Delta_g$ on $(X, g)$ (with Dirichlet or Neumann boundary conditions) with eigenvalue $\la \geq 1$, and any $r$ satisfying $\la^{-\frac12} \leq r <s/4$, we have $$||\psi_{\lambda} ||_{L^p(\Sigma \cap B_{r/2}(x))} \leq C_s r^{-\frac{1}{2}} || \psi_{\lambda}||_{L^2(B_{r}(x))} \lambda^{\delta(n, k, p)}. $$ Corollary \ref{BoundaryCor} follows by choosing $r =\frac{s}{5}$ and using a covering argument as in the proof of Theorem \ref{LpEstimatesSS}. Theorem \ref{LpQE} needs two additional ingredients: Lemma \ref{QElemma} and the QE result of \cite{ZZ}.

\section*{Acknowledgments}

Th author would like to thank Gabriel Rivi\`ere for their comments on the first draft.

\end{document}